\documentclass[reqno, dvipdfmx]{amsart}
\usepackage{amsmath} 

\usepackage{amsthm}
\usepackage{amssymb}
\usepackage{graphics}
\usepackage{latexsym}
\usepackage{comment}
\usepackage{extarrows}
\usepackage{wrapfig}
\usepackage{comment}
\usepackage {colortbl,array,xcolor}
\usepackage{hhline}
\usepackage{tikz}
\usepackage[mathlines,pagewise]{lineno}
\newcommand*\patchAmsMathEnvironmentForLineno[1]{%
  \expandafter\let\csname old#1\expandafter\endcsname\csname #1\endcsname
  \expandafter\let\csname oldend#1\expandafter\endcsname\csname end#1\endcsname
  \renewenvironment{#1}%
     {\linenomath\csname old#1\endcsname}%
     {\csname oldend#1\endcsname\endlinenomath}}% 
\newcommand*\patchBothAmsMathEnvironmentsForLineno[1]{%
  \patchAmsMathEnvironmentForLineno{#1}%
  \patchAmsMathEnvironmentForLineno{#1*}}%
\AtBeginDocument{%
\patchBothAmsMathEnvironmentsForLineno{equation}%
\patchBothAmsMathEnvironmentsForLineno{align}%
\patchBothAmsMathEnvironmentsForLineno{flalign}%
\patchBothAmsMathEnvironmentsForLineno{alignat}%
\patchBothAmsMathEnvironmentsForLineno{gather}%
\patchBothAmsMathEnvironmentsForLineno{multline}%
}
\usetikzlibrary{intersections, calc, math, patterns}

\numberwithin{equation}{section}
\newtheorem{thm}{Theorem}[section]
\newtheorem{prop}[thm]{Proposition}
\newtheorem{lem}[thm]{Lemma}

\theoremstyle{remark}
\newtheorem{rem}{Remark}[section]
\newtheorem{defn}{Definition}

\newcommand{\laplacian}{\Delta}
\newcommand{\R}{{\mathbb R}}

\newcommand{\N}{{\mathbb N}}
\newcommand{\C}{{\mathbb C}}
\newcommand{\LR}[1]{{\langle {#1} \rangle }}
\newcommand{\cross}{\times}
\newcommand{\e}{\varepsilon}

\newcommand{\F}{\mathcal{F}}

\newcommand{\ha}{\widehat}

\newcommand{\supp}{\operatorname{supp}}
\title[Well-posedness of the modified Zakharov-Kuznetsov equation]{Well-posedness\\ for 
the Cauchy problem of \\
the modified Zakharov-Kuznetsov equation
}

\author[S.  Kinoshita]{Shinya Kinoshita}
\address[Shinya Kinoshita]{Universit\"{a}t Bielefeld
Fakult\"{a}t f\"{u}r Mathematik
Postfach 10 01 31
33501 Bielefeld
Germany}
\email[Shinya Kinoshita]{kinoshita@math.uni-bielefeld.de}


\begin{document}

\begin{abstract}
This paper is concerned with the Cauchy problem of the modified Zakharov-Kuznetsov equation on $\R^d$. 
If $d=2$, we prove the sharp estimate which implies local in time well-posedness in the Sobolev space $H^s(\R^2)$ for $s \geq 1/4$. If $d \geq 3$, by employing $U^p$ and $V^p$ spaces, we establish the small data global well-posedness in the scaling critical Sobolev space $H^{s_c}(\R^d)$ where $s_c = d/2-1$.
\end{abstract}
\maketitle
\setcounter{page}{001}

%%%%%%%%%%%%%%%%%%%%%%%%%%%%%%%%%%%%%%%%%%%%%%%%%%%%%%%%%%%%%%%%%%%%%%%%%%%%%%%%%
%%%%%%%%%%%%%%%%%%%%%%%%%%%%%%%%%%%%%%%%%%%%%%%%%%%%%%%%%%%%%%%%%%%%%%%%%%%%%%%%%
%%%%%%%%%%%%%%%%%%%%%%%%%%%%%%%%%  Section 1  %%%%%%%%%%%%%%%%%%%%%%%%%%%%%%%%%%%%%
%%%%%%%%%%%%%%%%%%%%%%%%%%%%%%%%%%%%%%%%%%%%%%%%%%%%%%%%%%%%%%%%%%%%%%%%%%%%%%%%%
%%%%%%%%%%%%%%%%%%%%%%%%%%%%%%%%%%%%%%%%%%%%%%%%%%%%%%%%%%%%%%%%%%%%%%%%%%%%%%%%%

\section{Introduction}
We consider the Cauchy problem of the generalized Zakharov-Kuznetsov equation
\begin{equation}
 \begin{cases}
  \partial_t u + \partial_{x_1} \laplacian u =  \partial_{x_1} (u^{k+1}), \quad (t,x_1, \cdots, x_d) \in \R\cross \R^d, \\
  u(0, \cdot) = u_0 \in H^s(\R^d),\label{gZK}
 \end{cases}  
\end{equation}
where $d \geq 2$, $k \in \N$, 
$u= u(t,x_1, \cdots, x_d)$ is a real valued function and $\laplacian = \partial_{x_1}^2 +
\cdots + \partial_{x_d}^2$ is the Laplacian. 
When $k=1$, \eqref{gZK} is called the Zakharov-Kuznetsov equation which was introduced by Zakharov and Kuznetsov in \cite{ZK74} as a model for the propagation of ion-sound waves in magnetic fields for $d=3$. See also \cite{LS82}. In \cite{LLS13}, Lannes, Linares and Saut derived the Zakharov-Kuznetsov equation in dimensions $2$ and $3$ rigorously as a long-wave limit of the Euler-Poisson system. 
The generalized Zakharov-Kuznetsov equation can be seen as a multi-dimensional extension of the generalized KdV equation
\begin{equation*}
\partial_t u + \partial_{x}^3 u + u^k  \partial_{x} u = 0, \quad (t,x) \in \R \cross \R.
\end{equation*}
There are lots of works on the generalized Zakharov-Kuznetsov equation \eqref{gZK}. 
For $d=2$, we refer to the papers \cite{BL03}, 
\cite{Fa95}, \cite{GH14}, \cite{Ki2019}, \cite{LP09}, \cite{MP15} 
for the case $k=1$, \cite{BFR19}, \cite{BL03}, \cite{FLP11}, \cite{LP09}, \cite{LP11}, \cite{RV12} for the case $k=2$, and 
see \cite{FLP11}, \cite{Gru15}, \cite{LP11}, \cite{RV12} for $k \geq 3$. 
For $d=3$, we refer to \cite{LS09}, \cite{RV12-2} for the case $k=1$, and \cite{Gru14} for $k=2$, and  \cite{Gru15} for $k \geq 3$.

The aim of the paper is to establish well-posedness of \eqref{gZK} when $k=2$:
\begin{equation}
 \begin{cases}
  \partial_t u + \partial_{x_1} \laplacian u =  \partial_{x_1} (u^{3}), \quad (t,x_1, \cdots, x_d) \in \R\cross \R^d, \\
  u(0, \cdot) = u_0 \in H^s(\R^d),\label{mZK}
 \end{cases}  
\end{equation}
which we call the modified Zakharov-Kuznetsov equation (mZK). 
The paper is divided into two parts. The first part is devoted to the well-posedess of the $2$D (mZK) and in the latter part we consider \eqref{mZK} for $d \geq 3$. 
The main reason for this is that when $d=2$, by performing a linear change of variables as in \cite{GH14}, \eqref{mZK} can be rewritten as follows.
\begin{equation}
 \begin{cases}
  \partial_t v + (\partial_x^3 + \partial_y^3) v = 4^{-\frac{1}{3}} 
(\partial_x+ \partial_y) (v^3), \quad (t,x,y) \in \R \cross \R^2, \\
  v(0, \cdot) = v_0 \in H^s(\R^2).\label{mZK'}
 \end{cases}  
\end{equation}
This can be observed by putting $x= 4^{-1/3}x_1 + \sqrt{3} 4^{-1/3} x_2$, 
$y= 4^{-1/3}x_1 - \sqrt{3} 4^{-1/3} x_2$ and $v(t,x,y) := u(t,x_1,x_2)$, 
$v_0(x,y) := u_0(x_1,x_2)$. 
It is clear that the above linear transformation $(x_1,x_2) \to (x,y)$ is invertible as a mapping $\R^2 \to \R^2$, 
which means that the Cauchy problem \eqref{mZK'} is equivalent to \eqref{mZK} when $d=2$. 
We will see that, because of the symmetry of $x$ and $y$, it is convenient that we consider the symmetrized equation \eqref{mZK'} instead of \eqref{mZK}. While for $d \geq 3$ there is no transformation to symmetrize the modified Zakharov-Kuznetsov equation.

We now state the main results. 
\begin{thm}  \label{mth}
Let $d=2$ and $s \geq1/4$. Then the Cauchy problem \eqref{mZK'} is locally well-posed in $H^{s}(\R^2)$.
\end{thm}
\begin{thm}\label{mth2}
Let $d \geq 3$. Then the Cauchy problem \eqref{mZK} is small data globally well-posed in $H^{s_c}(\R^d)$.
\end{thm}
We give a comment on Theorem \ref{mth}. For $d=2$, in \cite{LP09}, Linares and Pastor proved the local well-posedness of (mZK) for $s > 3/4$. After that, 
the local well-posedness of the $2$D (mZK) for $s>1/4$, 
which is the best known result so far, was established by Ribaud and Vento in \cite{RV12}. 
The global results of $2$D (mZK) can be found in \cite{BFR19} and \cite{LP11}. 
When $d=2$, the scaling critical index $s_c$ of (mZK) is $0$. 
In \cite{LP09}, Linares and Pastor proved that \eqref{mZK'} is ill-posed in $H^s(\R^2)$ if $s \leq 0$ in the sense that the data-to-solution map fails to be uniformly continuous. 
As far as we know, there are no results for the case $0 < s \leq 1/4$. 
Theorem \ref{mth} establishes the well-posedness at $s=1/4$ which is in fact optimal for the Picard iteration approach, as the following theorem shows. 
\begin{thm}\label{not-c3}
Let $s <  1/4$. Then for any $T>0$, the data-to-solution map 
$ u_0 \mapsto u$ 
of \eqref{mZK'}, as a map from the unit ball in 
$H^s(\R^2)$ to 
$C([0,T]; H^{s})$ fails to be $C^3$.
\end{thm}
%%%%%%%%%%%%%%%%%%%%%%%%%%%%%%%%%%%%%%%%%%%%%%%%%%%%%%%%%
%%%%%%%%%%%%%%%%%%%%%%%%%%%%%%%%%%%%%%%%%%%%%%%%%%%%%%%%%
%%%%%%%%%%%%%%%%%%%%%%%%%%%%%%%%%%%%%%%%%%%%%%%%%%%%%%%%%
%%%%%%%%%%%%%%%%%%%%%%%%%%%%%%%%%%%%%%%%%%%%%%%%%%%%%%%%%
\begin{proof}
We follow the Bourgain's argument which was introduced in \cite{Bo97}. See also Section 6 in \cite{Ho07}. 
It should be noted that the function we choose below is essentially the same as the one which was employed to show the not-$C^3$ result of the modified KdV equation in \cite{Bo97}. 

It suffices to show that 
if $s<1/4$ for any $C>0$ there exists a real-valued function $\varphi \in \mathcal{S}(\R^2)$ such that 
$\|\varphi\|_{H^s(\R^2)} \sim 1$ and
\begin{equation}
  \left\| \int_0^t{e^{-(t-t')(\partial_\xi^3 + \partial_\eta^3)}  (\partial_x + \partial_y) \left( 
(e^{-t'(\partial_\xi^3 + \partial_\eta^3)}\varphi) \, (e^{-t'(\partial_\xi^3 + \partial_\eta^3)} \varphi) \, 
(e^{-t'(\partial_\xi^3 + \partial_\eta^3)} \varphi)\right) } 
d t'\right\|_{H^{s}}
 \geq C.\label{goal-thm1.3}
\end{equation}
Define real-valued even functions $\psi_{N,\xi}$, $\psi_{\eta} \in \mathcal{S}(\R)$ as
\begin{equation*}
\psi_{N,\xi}(\xi) = 
\begin{cases}
1 \ \ \textnormal{if} \ \ N \leq |\xi| \leq N+ N^{-{\frac{1}{2}}}\\
0 \ \ \textnormal{if} \ |\pm \xi-N - 2^{-1}N^{-\frac{1}{2}} | \geq N^{-\frac{1}{2}},
\end{cases}
\psi_{\eta}(\eta) = 
\begin{cases}
1 \ \ \textnormal{if} \ |\eta| \leq 1\\
0 \ \ \textnormal{if} \ |\eta | \geq 2,
\end{cases}
\end{equation*}
and $\varphi_N$ by $(\F_{x,y} \varphi_N )(\xi,\eta)=  N^{-s + 1/4}\psi_{N,\xi}(\xi) \psi_{\eta} (\eta)$. 
Then $\|\varphi_N \|_{H^s(\R^2)} \sim 1$. 
Let
\begin{equation*}
\Phi(\xi_1,\xi_2,\xi_3):= (\xi_1+\xi_2+\xi_3)^3-\xi_1^3 - \xi_2^3 - \xi_3^3.
\end{equation*}
We easily observe that if 
\begin{equation*} 
|\xi_1-N - 2^{-1}N^{-\frac{1}{2}} | < N^{-\frac{1}{2}}, \ 
|\xi_2-N - 2^{-1}N^{-\frac{1}{2}} | < N^{-\frac{1}{2}}, \ 
 |\xi_3+N + 2^{-1}N^{-\frac{1}{2}} | < N^{-\frac{1}{2}},
\end{equation*}
we have $|\Phi(\xi_1,\xi_2,\xi_3)| \lesssim 1$. Let $t$ be sufficiently small. 
By Plancherel's theorem, we get
\begin{align*}
&  \left\| \int_0^t{e^{-(t-t')(\partial_\xi^3 + \partial_\eta^3)}  (\partial_x + \partial_y) \left( 
(e^{-t'(\partial_\xi^3 + \partial_\eta^3)}\varphi_N) \, (e^{-t'(\partial_\xi^3 + \partial_\eta^3)} \varphi_N) \, 
(e^{-t'(\partial_\xi^3 + \partial_\eta^3)} \varphi_N)\right) } 
d t'\right\|_{H^{s}}\\
& \gtrsim N^{-2 s + 7/4} \biggl\| \chi_{\supp \psi_{N,\xi}}(\xi)
\int_0^t e^{- i t' (\Phi (\xi_1, \xi_2-\xi_1, \xi-\xi_2) +\Phi (\eta_1, \eta_2-\eta_1, \eta-\eta_2)) } \\
& \quad \ \iint \psi_{N,\xi}(\xi_1) \psi_{N,\xi}(\xi_2-\xi_1) \psi_{N,\xi}(\xi-\xi_2) d \xi_1 d\xi_2 
\iint \psi_{\eta}(\eta_1) \psi_{\eta}(\eta_2-\eta_1) \psi_{\eta}(\eta-\eta_2) d \eta_1 d\eta_2 
d t'\biggr\|_{L^2_{\xi\eta}}\\
& \gtrsim N^{-2 s +1/2}.
\end{align*}
This completes the proof of \eqref{goal-thm1.3}.
\end{proof}
%%%%%%%%%%%%%%%%%%%%%%%%%%%%%%%%%%%%%%%%%%%%%%%%%%%%%%%%%
%%%%%%%%%%%%%%%%%%%%%%%%%%%%%%%%%%%%%%%%%%%%%%%%%%%%%%%%%
%%%%%%%%%%%%%%%%%%%%%%%%%%%%%%%%%%%%%%%%%%%%%%%%%%%%%%%%%
%%%%%%%%%%%%%%%%%%%%%%%%%%%%%%%%%%%%%%%%%%%%%%%%%%%%%%%%%
%%%%%%%%%%%%%%%%%%%%%%%%%%%%%%%%%%%%%%%%%%%%%%%%%%%%%%%%%
%%%%%%%%%%%%%%%%%%%%%%%%%%%%%%%%%%%%%%%%%%%%%%%%%%%%%%%%%
%%%%%%%%%%%%%%%%%%%%%%%%%%%%%%%%%%%%%%%%%%%%%%%%%%%%%%%%%
%%%%%%%%%%%%%%%%%%%%%%%%%%%%%%%%%%%%%%%%%%%%%%%%%%%%%%%%%
Next we comment on Theorem \ref{mth2}. 
For the $3$D (mZK), in \cite{Gru14}, Gr\"{u}nrock established the local well-posedness in the full subcritical regime $s>1/2$. 
Theorem \ref{mth2} is an extension of the result by Gr\"{u}nrock. 
To be specific, Theorem \ref{mth2} establishes the small data global well-posedness in the scaling critical regularity Sobolev space for $d \geq 3$. 
The key ingredient in the proof of Theorem \ref{mth2} is that we employ $U^p$, $V^p$ spaces which were introduced by Koch and Tataru in \cite{KT05} and \cite{KT07}. See also \cite{HHK09} and \cite{HHK09-2}.
%%%%%%%%%%%%%%%%%%%%%%%%%%%%%%%%%%%%%%%%%%%%%%%%%%%%%%%%%
%%%%%%%%%%%%%%%%%%%%%%%%%%%%%%%%%%%%%%%%%%%%%%%%%%%%%%%%%
%%%%%%%%%%%%%%%%%%%%%%%%%%%%%%%%%%%%%%%%%%%%%%%%%%%%%%%%%
%%%%%%%%%%%%%%%%%%%%%%%%%%%%%%%%%%%%%%%%%%%%%%%%%%%%%%%%%
%%%%%%%%%%%%%%%%%%%%%%%%%%%%%%%%%%%%%%%%%%%%%%%%%%%%%%%%%

The paper is organized as follows. 
In Section 2, we introduce notations, $X^{s,b}$ space and estimates for the proof of Theorem \ref{mth}. 
Section 3 is devoted to the proof of the key estimate which establishes Theorem \ref{mth} immediately. 
In Sections 4 and 5, we consider Theorem \ref{mth2}. 
In the former section, we introduce $U^p$ and $V^p$ spaces and fundamental estimates. 
Lastly, in Section 4,  we will prove the key estimate which immediately provedes Theorem \ref{mth2}.
%%%%%%%%%%%%%%%%%%%%%%%%%%%%%%%%%%%%%%%%%%%%%%%%%%%%%%%%%
%%%%%%%%%%%%%%%%%%%%%%%%%%%%%%%%%%%%%%%%%%%%%%%%%%%%%%%%%
%%%%%%%%%%%%%%%%%%%%%%%%%%%%%%%%%%%%%%%%%%%%%%%%%%%%%%%%%
%%%%%%%%%%%%%%%%%%%%%%%%%%%%%%%%%%%%%%%%%%%%%%%%%%%%%%%%%
%%%%%%%%%%%%%%%%%%%%%%%%%%%%%%%%%%%%%%%%%%%%%%%%%%%%%%%%%

\vspace{5mm}
Throughout the paper, we use the following notations. 
$A{\ \lesssim \ } B$ means that there exists $C>0$ such that $A \le CB.$ 
Also, $A\sim B$ means $A{\ \lesssim \ } B$ and $B{\ \lesssim \ } A.$ 
Let $N$, $L \geq 1$ be dyadic numbers, i.e. there exist $n_1$, $n_2 \in \N_{0}$ such that 
$N= 2^{n_1}$ and $L=2^{n_2}$, and $\psi \in C^{\infty}_{0}((-2,2))$ be an even, non-negative function which satisfies $\psi (t)=1$ for $|t|\leq 1$ and letting 
$\psi_N (t):=\psi (t N^{-1})-\psi (2t N^{-1})$, 
$\psi_1(t):=\psi (t)$, 
the equality $\displaystyle{\sum_{N}\psi_{N}(t)=1}$ holds. 
Here we used $\displaystyle{\sum_{N}= \sum_{N \in 2^{\N_0}}}$ for simplicity. 
We also use the notations $\displaystyle{\sum_{L}= \sum_{L \in 2^{\N_0}}}$ and 
 $\displaystyle{\sum_{N,L}= \sum_{N,L \in 2^{\N_0}}}$ throughout the paper. 
\section*{Acknowledgement}
The author would like to thank Sebastian Herr for leading to the problem and giving him valuable suggestions. 
This work was supported by the DFG through the CRC 1283 ``Taming uncertainty and profiting from randomness and low regularity in analysis, stochastics and their applications.'' 
%%%
%%%
%%section2
%%%
%%%
\section{Preliminaries for Theorem \ref{mth}}
In this section, we introduce notations and estimates which will be utilized to establish the key bilinear estimate for Theorem \ref{mth}. 
Let $u=u(t,x,y)$ with $(t,x,y) \in \R \times \R^2$. 
$\F_t u,\ \F_{x,y} u$ denote the Fourier transform of $u$ in time, space, respectively. 
$\F_{t, x,y} u = \ha{u}$ denotes the Fourier transform of $u$ in space and time. 
We define frequency and modulation projections $P_N$, $Q_L$ as
\begin{align*}
(\F_{x,y}^{-1}P_{N}u )(\xi, \eta ):= & \psi_{N}(|(\xi, \eta)| )(\F_{x,y}{u})(\xi,\eta ),\\
\widehat{Q_{L} u}(\tau ,\xi, \eta ):= & \psi_{L}(\tau - \xi^3 - \eta^3)\widehat{u}(\tau ,\xi, \eta ).
\end{align*}
Let $s, b \in \R$. We define $X^{s,b}(\R^3)$ spaces. 
\begin{align*}
&  X^{s,\,b} (\R^3) :=\{ f \in \mathcal{S}'(\R^3) \ | \ \|f \|_{X^{s,\,b}} < \infty \},\\
& \|f  \|_{X^{s,\, b}} := \left( \sum_{N,\, L} N^{2s} L^{2b} \|  P_N Q_L f \|^2_{L_{x, t}^{2}} \right)^{1/2}.
\end{align*}
For convenience, we define the set in frequency as
\begin{equation*}
G_{N, L} := \{ (\tau, \xi, \eta) \in \R^3 \, | \, \psi_{L}(\tau - \xi^3 - \eta^3) \psi_{N}(|(\xi, \eta)| ) \not= 0 .\}
\end{equation*}
Next we observe the fundamental properties of $X^{s,\,b}$. 
A simple calculation gives the following.
\begin{equation*}
(i)  \ \ 
\overline{{X}^{s,\,b}} = X^{s,\,b} , \qquad 
(ii)  \ \ (X^{s,\,b})^* =X^{-s,\,-b},
\end{equation*}
for $s$, $b \in \R$.

Recall the Strichartz estimates for the unitary group $\{ e^{-t (\partial_x^3+ \partial_y^3)}\}$.
%%%%%%%%%%%%%%%%%%%
%%%%%%%%%%%%%%%%%
%%%%%%%%%%%%%%%%%%%%
%%%%%%%%%%%%%%%%%%%
%%%%%%%%%%%%%%%%%%%%%
%%%%%%%%%%%%%%%%%%%
%%%%%%%%%%%%%%
%
%
%
%
%%%%%%%%%%%%%%%%%%%%%%%%%%%%
\begin{lem}[Theorem 3.1. \cite{KPV91}]\label{thm2.1}
Let $\varphi \in L^2(\R^2)$. Then we have
\begin{align}
\| |\nabla_x|^{\frac{1}{2p}} |\nabla_y|^{\frac{1}{2p}} e^{-t (\partial_x^3+ \partial_y^3)} \varphi 
\|_{L_t^p L_{x,y}^q} & \lesssim 
\|\varphi\|_{L^2_{x,y}}, \quad 
\textit{if} \ \ \frac{2}{p} + \frac{2}{q} = 1, \ p >2,\label{Strichartz-01}\\
\|  e^{-t (\partial_x^3+ \partial_y^3)} \varphi  \|_{L_t^p L_{x,y}^q} & \lesssim \|\varphi\|_{L^2_{x,y}}, \quad 
\textit{if} \ \ \frac{3}{p} + \frac{2}{q} = 1, \ p >3,\label{Strichartz-02}
\end{align}
where $|\nabla_x|^s := \F^{-1}_x |\xi|^s \F_x$ and $|\nabla_y|^s :=  \F^{-1}_y |\eta|^s \F_y$ denote the Riesz potential operators with respect to $x$ and $y$, respectively.
\end{lem}
The Strichartz estimates above provide the following estimates. See \cite{GTV}. 
\begin{align}
\| |\nabla_x|^{\frac{1}{2p}} |\nabla_y|^{\frac{1}{2p}} Q_L u \|_{L_t^p L_{x,y}^q} & \lesssim L^{\frac{1}{2}} 
\| Q_L u\|_{L^2_{x,y,t}}, \quad 
\textnormal{if} \ \ \frac{2}{p} + \frac{2}{q} = 1, \ p >2,\label{Strichartz-1}\\
\|  Q_L u \|_{L_t^p L_{x,y}^q} & \lesssim L^{\frac{1}{2}} \| Q_L u\|_{L^2_{x,y,t}}, \quad 
\textnormal{if} \ \ \frac{3}{p} + \frac{2}{q} = 1, \ p >3.\label{Strichartz-2}
\end{align}
\begin{rem}
Since the estimates \eqref{Strichartz-1} and \eqref{Strichartz-2} are almost equivalent to \eqref{Strichartz-01} and \eqref{Strichartz-02}, respectively, 
we frequently call \eqref{Strichartz-1} and \eqref{Strichartz-2} Strichartz estimates in the paper.
\end{rem}
Next we introduce the bilinear transversal inequality. 
For $f :\R^3 \to \C$, we use the following notation hereafter.
\begin{equation*}
\supp_{\xi,\eta} f := \{(\xi,\eta) \in \R^2 \, | \, \textnormal{There exists $\tau \in \R$ such that } 
(\tau,\xi,\eta) \in \supp f. \}
\end{equation*}
\begin{prop}\label{bilinear-transversal}
Let $N_2 \leq N_1$, $\varphi (\xi,\eta) = \xi^3 + \eta^3$. 
Suppose that
\begin{equation*}
\supp \widehat{u}_{N_1,L_1} \subset G_{N_1,L_1}, \quad 
\supp \widehat{v}_{N_2,L_2} \subset G_{N_2,L_2},
\end{equation*}
and
\begin{equation*}
|\nabla \varphi(\xi_1,\eta_1)-\nabla \varphi(\xi_2,\eta_2)| \gtrsim N_1^2,
\end{equation*} 
for all $(\xi_1, \eta_1) \in \supp_{\xi,\eta} \widehat{u}_{N_1,L_1}$, 
$(\xi_2,\eta_2) \in \supp_{\xi,\eta} \widehat{v}_{N_2,L_2}$. 
Then we have
\begin{equation}
\| u_{N_1,L_1} \, v_{N_2, L_2} \|_{L_t^2 L_{xy}^2} \lesssim N_1^{-1} N_2^{\frac{1}{2}} 
(L_1 L_2)^{\frac{1}{2}} \| u_{N_1, L_1} \|_{L_t^2 L_{xy}^2} 
\|v_{N_2, L_2}  \|_{L_t^2 L_{xy}^2}.\label{strichartz-goal01}
\end{equation}
In particular, if $N_2 \leq 2^{-3}N_1$ and 
\begin{equation*}
\supp \widehat{u}_{N_1,L_1} \subset G_{N_1,L_1}, \quad 
\supp \widehat{v}_{N_2,L_2} \subset G_{N_2,L_2},
\end{equation*}
we have
\begin{equation}
\| u_{N_1,L_1} \, v_{N_2, L_2} \|_{L_t^2 L_{xy}^2} \lesssim N_1^{-1} N_2^{\frac{1}{2}}
(L_1 L_2)^{\frac{1}{2}} \| u_{N_1, L_1} \|_{L_t^2 L_{xy}^2} \|v_{N_2, L_2}  \|_{L_t^2 L_{xy}^2}.\label{strichartz-goal02}
\end{equation}
\end{prop}
%%%%%%%%%%%%%%%%%%%%%%%%%%%%%%%%%%%%%%%%%%%%%%%%%%%%%%%%%
%%%%%%%%%%%%%%%%%%%%%%%%%%%%%%%%%%%%%%%%%%%%%%%%%%%%%%%%%
\begin{proof}
First we consider \eqref{strichartz-goal01}. 
By Plancherel's theorem, it suffices to show
\begin{equation}
\begin{split}
& \left\| \int  \ha{u}_{N_1, L_1}(\tau_1, \xi_1,\eta_1) 
\ha{v}_{N_2, L_2} (\tau- \tau_1, \xi-\xi_1, \eta-\eta_1 ) d\tau_1 d \xi_1 d\eta_1 \right\|_{L^2}  \\ 
& \qquad \qquad \qquad \qquad \qquad 
\lesssim N_1^{-1} N_2^{\frac{1}{2}}
(L_1 L_2)^{\frac{1}{2}} \| \widehat{u}_{N_1, L_1} \|_{L^2} \| \widehat{v}_{N_2, L_2}  \|_{L^2}.
\end{split}\label{est03-bilinear-transversal}
\end{equation}
By the almost orthogonality, we may assume that 
$\supp_{\xi,\eta} \ha{u}_{N_1, L_1}$ and $\supp_{\xi,\eta} \ha{v}_{N_2, L_2}$ are confined to balls whose radius $r$ such that $r \ll N_2$, respectively. 
Since $\varphi$ is a cubic polynomial, we deduce from $N_2 \leq N_1$ that 
\begin{equation*}
\sup_{1 \leq i,j \leq 2}(|\partial_i \partial_j \varphi(\xi_1,\eta_1)| + 
|\partial_i \partial_j \varphi(\xi- \xi_1,\eta-\eta_1)|) \lesssim N_1.
\end{equation*} 
Therefore, we easily observe 
\begin{equation*}
|\nabla \varphi(\xi,\eta)-\nabla \varphi(\xi',\eta')| \ll N_1 N_2 \quad \textnormal{if } \ 
|(\xi,\eta) - (\xi', \eta')| \ll N_2.
\end{equation*}
This implies that, without loss of generality, we may assume that
\begin{equation}
|\partial_1 \varphi(\xi_1,\eta_1)-\partial_1 \varphi(\xi_2,\eta_2)| 
\gtrsim N_1^2,\label{est04-bilinear-transversal}
\end{equation} 
for all  $(\xi_1, \eta_1) \in \supp_{\xi,\eta} \widehat{u}_{N_1,L_1}$, 
$(\xi_2,\eta_2) \in \supp_{\xi,\eta} \widehat{v}_{N_2,L_2}$. 
Now we turn to \eqref{est03-bilinear-transversal}. By the Cauchy-Schwarz inequality, we get
\begin{align*}
\biggl\| \int  \ha{u}_{N_1, L_1}(\tau_1, \xi_1,\eta_1)  &
\ha{v}_{N_2, L_2}  (\tau- \tau_1, \xi-\xi_1,\eta-\eta_1 ) d\tau_1 d \xi_1 d \eta_1 \biggr\|_{L^2}  \\
\leq & \left\|   \left(\left| 
\ha{u}_{N_1, L_1}  \right|^2 * 
\left|\ha{v}_{N_2, L_2}  \right|^2
 \right)^{1/2} |E(\tau, \xi,\eta )|^{1/2} \right\|_{L^2} \\
\leq & \sup_{\tau, \xi,\eta} |E(\tau, \xi,\eta )|^{1/2} 
 \left\| \left| 
 \ha{u}_{N_1, L_1} \right|^2 * 
\left|\ha{v}_{N_2, L_2}  \right|^2
\right\|_{L^1}^{1/2}\\
\leq & \sup_{\tau, \xi,\eta} |E(\tau, \xi,\eta )|^{1/2} 
\|\ha{u}_{N_1, L_1}\|_{L^2} 
\| \ha{v}_{N_2, L_2}  \|_{L^2},
\end{align*}
where $E(\tau, \xi,\eta) \subset \R^{3}$ is defined by
\begin{equation*}
E(\tau, \xi,\eta) := \{ (\tau_1, \xi_1,\eta_1) \in \supp \widehat{u}_{N_1,L_1}
\, | \, (\tau-\tau_1, \xi- \xi_1,\eta-\eta_1) \in \widehat{v}_{N_2,L_2} \}.
\end{equation*}
Thus, it suffices to show
\begin{equation}
|E(\tau, \xi,\eta)| \lesssim N_1^{-2} N_2 L_1 L_2.\label{est05-bilinear-transversal}
\end{equation}
If we fix $(\xi_1,\eta_1)$, it is easily observed that
\begin{equation}
 | \{ \tau_1 \, | \, (\tau_1, \xi_1,\eta_1) \in E(\tau, \xi,\eta) \}| 
\lesssim \min(L_1, L_2).\label{est06-bilinear-transversal}
\end{equation}
Next, if we fix 
$\eta_1$, since $\max (L_1, L_2)  \gtrsim  |\varphi(\xi_1,\eta_1)+\varphi(\xi-\xi_1,\eta-\eta_1)-\tau|$, 
the inequality \eqref{est04-bilinear-transversal} implies that $\xi_1$ is confined to an interval whose length is comparable to $\max (L_1, L_2)/N_1^2$. 
This, combined with \eqref{est06-bilinear-transversal} and 
$(\tau_1, \xi_1,\eta_1) \in \supp \widehat{u}_{N_1,L_1}$ which implies $\mathcal{O}(\eta_1) \leq N_2$, 
yields \eqref{est05-bilinear-transversal}.

To see \eqref{strichartz-goal02}, it suffices to show
\begin{equation*}
|(\xi_1,\eta_1)| \geq 2 |(\xi_2,\eta_2)| \Longrightarrow 
|\nabla \varphi(\xi_1,\eta_1)-\nabla \varphi(\xi_2,\eta_2)| \gtrsim |(\xi_1,\eta_1)|^2,
\end{equation*}
which is verified by a simple calculation.
\end{proof}
%%%%%%%%%%%%%%%%%%%%%%%%%%%%%%%%%%%%%%%%%%%%%%%%%%%%%%%%%
%%%%%%%%%%%%%%%%%%%%%%%%%%%%%%%%%%%%%%%%%%%%%%%%%%%%%%%%%
%%%%%%%%%%%%%%%%%%%%%%%%%%%%%%%%%%%
%%%%%%%%%%%%%%%%%%%%%%%%%%%%%%%%%%%
%%%%%%%section2%%%%%%%%%%%%%%%%%%%%
%%%%%%%%%%%%%%%%%%%%%%%%%%%%%%%%%%%
%%%%%%%%%%%%%%%%%%%%%%%%%%%%%%%%%%
\section{Proof of the Key estimate for Theorem \ref{mth}}
In this section, we establish the key estimate which gives Theorem \ref{mth} by a standard iteration argument, see \cite{GTV}, \cite{KPV96}, and \cite{Tao06}. 
In this paper, we omit the details of the proof of Theorem \ref{mth} and focus on showing the following key estimate.
\begin{thm}\label{nonlinearity-estimate}
For any $s \geq 1/4$, there exist $b \in (1/2, 1)$, $\e >0$ and $C>0$ such that
\begin{equation}
\|(\partial_x + \partial_y) (u_1 u_2 u_3) \|_{X^{s,\,b-1+\e}}   \leq C \prod_{i=1}^3 \|u_i \|_{X^{s,\,b}}.
\label{goal01-mth}
\end{equation}
\end{thm}
By a duality argument and dyadic decompositions, we observe that 
\begin{align}
(\ref{goal01-mth}) 
\iff \ & \left| \int{ u_4 (\partial_x + \partial_y) (u_1 u_2 u_3)} dtdxdy \right| 
\lesssim \prod_{i=1}^3 \|u_i \|_{X^{s,\,b}} \| u_4\|_{X^{-s, 1-b-\e}}.\notag \\
\Longleftarrow \ \, & \sum_{{{\substack{N_i, L_i\\ \tiny{(i=1,2,3,4)}}}}} N_4 
 \left|\int{ \prod_{i=1}^4 (Q_{L_i} P_{N_i}u_i) 
} 
dt dx dy \right|\lesssim \prod_{i=1}^3 
\|u_i \|_{X^{s,\,b}} \| u_4\|_{X^{-s, 1-b-\e}}.\label{goal02-mth}
\end{align}
For simplicity, we use the following notations.
\begin{align*}
& N_{\min} = \min (N_1, N_2, N_3,N_4), \quad  
N_{\max} = \max (N_1, N_2, N_3,N_4),\\
& L_{\max} = \max (L_1, L_2, L_3,L_4), \quad \ u_{N_i, L_i}= Q_{L_i} P_{N_i}u_i. 
\end{align*}
Clearly, \eqref{goal02-mth} is verified by showing
\begin{equation}
\left| \int{ u_{N_1, L_1} u_{N_2, L_2} u_{N_3, L_3} u_{N_4, L_4}} dtdxdy \right| 
\lesssim N_{\min}^{\frac{3}{4}} N_{\max}^{-\frac{5}{4}} (L_1 L_2 L_3 L_4)^{\frac{1}{2}- \e}
\prod_{i=1}^{4} \|u_{N_i, L_i} \|_{L^2}.\label{goal03-mth}
\end{equation}
By symmetry, we assume $N_1 \gtrsim N_2 \gtrsim N_3 \gtrsim N_4$. 
We first note that if $N_1 \sim 1$ we easily obtain \eqref{goal03-mth} by using the Strichartz estimates. 
Further, by the Strichartz estimates, we can see that $L_{\max} \gtrsim N_1^3$ yields \eqref{goal03-mth}. 
For simplicity, here we only treat the case $L_4 = L_{\max}$. 
The other cases can be treated in the same way. 
By the H\"{o}lder's inequality, the Strichartz estimates \eqref{Strichartz-2} with $(p,q)=(6,4)$ and the 
Sobolev inequality, we get
\begin{align*}
& \left| \int{ u_{N_1, L_1} u_{N_2, L_2} u_{N_3, L_3} u_{N_4, L_4}} dtdxdy \right| \\
& \lesssim \|u_{N_1,L_1}\|_{L_t^6 L_{xy}^4} \|u_{N_2,L_2}\|_{L_t^6 L_{xy}^4} \|u_{N_2,L_2}\|_{L_t^6 L_{xy}^4} 
\|u_{N_4,L_4}\|_{L_t^2 L_{xy}^4} \\
& \lesssim (L_1 L_2 L_3)^{\frac{1}{2}}  \|u_{N_1,L_1}\|_{L^2} \|u_{N_2,L_2}\|_{L^2} \|u_{N_2,L_2}\|_{L^2}
N_4^{\frac{1}{2}}\|u_{N_4,L_4}\|_{L^2} \\
& \lesssim N_4^{\frac{1}{2}}N_1^{-\frac{5}{4}} (L_1 L_2 L_3)^{\frac{1}{2}} L_4^{\frac{5}{12}}\prod_{i=1}^{4} \|u_{N_i, L_i} \|_{L^2}.
\end{align*}
This completes the proof of \eqref{goal03-mth}. 

%%%%%%%%%%%%%%%%%%%%%%%%%%%%%%%%%%%%%%%%%%%%%%%%%%%%%%%%%
%%%%%%%%%%%%%%%%%%%%%%%%%%%%%%%%%%%%%%%%%%%%%%%%%%%%%%%%%
Hereafter, we assume $1 \ll N_1$ and $L_{\max} \ll N_1^3$. 
We divide the proof into the following three cases.

Case 1: $N_1 \sim N_2 \gg N_3 \gtrsim N_4$,

Case 2: $N_1 \sim N_2 \sim N_3 \gg N_4$,

Case 3: $N_1 \sim N_2 \sim N_3 \sim N_4$.\\
\underline{Case 1: $N_1 \sim N_2 \gg N_3 \gtrsim N_4$.} 
Since $N_1 \gg N_3$ and $N_2 \gg N_4$, this case is easily handled by the bilinear transversal estimate 
\eqref{strichartz-goal02} as follows.
\begin{align*}
& \left| \int{ u_{N_1, L_1} u_{N_2, L_2} u_{N_3, L_3} u_{N_4, L_4}} dtdxdy \right| \\
& \lesssim \|u_{N_1,L_1} u_{N_3, L_3}\|_{L_t^2 L_{xy}^2} \|u_{N_2,L_2}u_{N_4,L_4}\|_{L_t^2 L_{xy}^2}\\
& \lesssim N_1^{-2} N_3^{\frac{1}{2}} N_4^{\frac{1}{2}} (L_1 L_2 L_3 L_4)^{\frac{1}{2}} \prod_{i=1}^{4} \|u_{N_i, L_i} \|_{L^2},
\end{align*}
which completes the proof of \eqref{goal03-mth}.\\
\underline{Case 2: $N_1 \sim N_2 \sim N_3 \gg N_4$.} 
By harmless decompositions, we may assume that 
$\supp_{\xi,\eta} \widehat{u}_{N_j,L_j}$ 
$(j=1,2,3)$ is contained in a ball such that its radius $r$ satisfies $r \ll N_1$. 
We divide the proof into two cases. 
First we assume 
\begin{equation*}
|\nabla \varphi(\xi_1, \eta_1) - \nabla \varphi (\xi_2,\eta_2)| \gtrsim N_1^2,
\end{equation*}
for all $(\xi_1, \eta_1) \in \supp_{\xi,\eta} \widehat{u}_{N_1,L_1}$, 
$(\xi_2,\eta_2) \in \supp_{\xi,\eta} \widehat{u}_{N_2,L_2}$. 
In this case, \eqref{strichartz-goal01} in Proposition \ref{bilinear-transversal} gives
\begin{equation}
\| u_{N_1,L_1} \, u_{N_2, L_2} \|_{L_t^2 L_{xy}^2} \lesssim N_1^{-\frac{1}{2}} 
(L_1 L_2)^{\frac{1}{2}} \| u_{N_1, L_1} \|_{L_t^2 L_{xy}^2} 
\|u_{N_2, L_2}  \|_{L_t^2 L_{xy}^2}.\label{est01-case2}
\end{equation}
On the other hand, since $N_1 \sim N_3 \gg N_4$, we get
\begin{equation}
\| u_{N_3,L_3} \, u_{N_4, L_4} \|_{L_t^2 L_{xy}^2} \lesssim N_1^{-1} N_4^{\frac{1}{2}}
(L_3 L_4)^{\frac{1}{2}} \| u_{N_3, L_3} \|_{L_t^2 L_{xy}^2} \|u_{N_4, L_4}  \|_{L_t^2 L_{xy}^2}.\label{est02-case2}
\end{equation}
Consequently, by the H\"{o}lder's inequality and \eqref{est01-case2}, \eqref{est02-case2}, we have
\begin{align*}
& \left| \int{ u_{N_1, L_1} u_{N_2, L_2} u_{N_3, L_3} u_{N_4, L_4}} dtdxdy \right| \\
& \lesssim \|u_{N_1,L_1} u_{N_2, L_2}\|_{L_t^2 L_{xy}^2} \|u_{N_3,L_3}u_{N_4,L_4}\|_{L_t^2 L_{xy}^2}\\
& \lesssim N_1^{-\frac{3}{2}} N_4^{\frac{1}{2}} (L_1 L_2 L_3 L_4)^{\frac{1}{2}} \prod_{i=1}^{4} \|u_{N_i, L_i} \|_{L^2},
\end{align*}
which completes the proof of \eqref{goal03-mth}. 
Next suppose that there exist $(\xi_1, \eta_1) \in \supp_{\xi,\eta} \widehat{u}_{N_1,L_1}$, 
$(\xi_2,\eta_2) \in \supp_{\xi,\eta} \widehat{u}_{N_2,L_2}$ such that 
\begin{equation}
|\nabla \varphi(\xi_1, \eta_1) - \nabla \varphi (\xi_2,\eta_2)| \ll N_1^2.\label{est03-case2}
\end{equation}
Since $|(\xi_1, \eta_1)| \geq N_1/2$, without loss of generality, we can assume $|\xi_1| \geq N_1/4$. 
This and \eqref{est03-case2} imply
\begin{align*}
|\partial_{1} & \varphi(\xi_1,\eta_1) - \partial_{1} \varphi(\xi_2,\eta_2)| \ll N_1^2 \\
& \iff 3|(\xi_1 - \xi_2)(\xi_1 + \xi_2)| \ll N_1^2\\
& \ \Longrightarrow  \ |\xi_2(2 \xi_1 + \xi_2)| \gtrsim N_1^2\\
&  \ \Longrightarrow \ |\partial_{1} \varphi(\xi_1,\eta_1) - \partial_{1} \varphi(\xi_1+\xi_2,\eta_1+\eta_2)|
\gtrsim N_1^2.
\end{align*}
Thus, because $N_1 \gg N_4$, we may assume
\begin{equation*}
|\nabla \varphi(\xi_1, \eta_1) - \nabla \varphi (\xi_3,\eta_3)| \gtrsim N_1^2,
\end{equation*}
for all $(\xi_1, \eta_1) \in \supp_{\xi,\eta} \widehat{u}_{N_1,L_1}$, 
$(\xi_3,\eta_3) \in \supp_{\xi,\eta} \widehat{u}_{N_3,L_3}$. 
Therefore, in the same manner as for the former case, Proposition \ref{bilinear-transversal} provides
\begin{align*}
& \left| \int{ u_{N_1, L_1} u_{N_2, L_2} u_{N_3, L_3} u_{N_4, L_4}} dtdxdy \right| \\
& \lesssim \|u_{N_1,L_1} u_{N_3, L_3}\|_{L_t^2 L_{xy}^2} \|u_{N_2,L_2}u_{N_4,L_4}\|_{L_t^2 L_{xy}^2}\\
& \lesssim N_1^{-\frac{3}{2}} N_4^{\frac{1}{2}} (L_1 L_2 L_3 L_4)^{\frac{1}{2}} \prod_{i=1}^{4} \|u_{N_i, L_i} \|_{L^2}.
\end{align*}
\underline{Case 3: $N_1 \sim N_2 \sim N_3 \sim N_4$.} 
Similarly to the previous case, by performing harmless decompositions, 
we assume that $\supp_{\xi,\eta} \widehat{u}_{N_i,L_i}$ 
$(i=1,2,3,4)$ is contained in a ball whose radius $r$ satisfies $r \ll N_1$. 
First we deal with the simple case $|\xi_i|\sim |\eta_i| \sim N_1$ $(i=1,2,3,4)$. 
By employing the Strichartz estimate \eqref{Strichartz-1} with $p=q=4$, we have
\begin{equation}
\|u_{N_i,L_i} \|_{L_t^4 L_{xy}^4} \lesssim  N_1^{-\frac{1}{4}}L_i^{\frac{1}{2}} \|u_{N_i,L_i}\|_{L^2},
\label{est01-case3-0}
\end{equation}
which immediately yields \eqref{goal03-mth} as follows.
\begin{align*}
 \left| \int{ u_{N_1, L_1} u_{N_2, L_2} u_{N_3, L_3} u_{N_4, L_4}} dtdxdy \right| 
& \lesssim \prod_{i=1}^4 \|u_{N_i,L_i}\|_{L^4}\\
& \lesssim N_1^{-1}  (L_1 L_2 L_3 L_4)^{\frac{1}{2}} \prod_{i=1}^{4} \|u_{N_i, L_i} \|_{L^2}.
\end{align*}
Thus, without loss of generality, we may assume that $|\eta_4| \ll N_1$. 
We divide the proof into three cases.

(1) $\min(|\xi_1|, |\xi_2|, |\xi_3|) \ll N_1$,

(2) $|\xi_1| \sim |\xi_2| \sim |\xi_3| \sim N_1$, $\ $ 
$L_{\max}^{8\e} \leq \max(|\eta_1|, |\eta_2|, |\eta_2|, |\eta_4|)$,

(3) $|\xi_1| \sim |\xi_2| \sim |\xi_3| \sim N_1$, $\ $ 
$\max(|\eta_1|, |\eta_2|, |\eta_2|, |\eta_4|) \leq L_{\max}^{8\e}.$\\

%%%%%%%%%%%%%%%%%%%%%%%%%%%%%%%%%%%%%%%%%%%%%%%%%%%%%%%%%
%%%%%%%%%%%%%%%%%%%%%%%%%%%%%%%%%%%%%%%%%%%%%%%%%%%%%%%%%
We consider the first case. Without loss of generality, we can assume $|\xi_3| \ll N_1$. 
Note that, since $N_1 \sim N_3 \sim N_4$, 
it holds that $|\eta_3| \sim |\xi_4| \sim N_1$, and then it is easily obtained 
$|\nabla \varphi(\xi_3, \eta_3) - \nabla \varphi (\xi_4,\eta_4)| \gtrsim N_1^2$ which, by utilizing Proposition \ref{bilinear-transversal}, yields
\begin{equation}
\| u_{N_3,L_3} \, u_{N_4, L_4} \|_{L_t^2 L_{xy}^2} \lesssim N_1^{-\frac{1}{2}}
(L_3 L_4)^{\frac{1}{2}} \| u_{N_3, L_3} \|_{L_t^2 L_{xy}^2} \|u_{N_4, L_4}  \|_{L_t^2 L_{xy}^2}.\label{est01-case3-1}
\end{equation}
We consider two cases. First assume that $|\xi_1|\sim |\eta_1| \sim |\xi_2|\sim |\eta_2|\sim N_1$. 
Recall that this condition provides the $L^4$ Strichartz estimate \eqref{est01-case3-0} for $i=1,2$. 
Then by \eqref{est01-case3-1}, we have
\begin{align*}
 \left| \int{ u_{N_1, L_1} u_{N_2, L_2} u_{N_3, L_3} u_{N_4, L_4}} dtdxdy \right| 
& \lesssim  \|u_{N_1,L_1}\|_{L^4} \|u_{N_2,L_2}\|_{L^4} \| u_{N_3,L_3} \, u_{N_4, L_4} \|_{L^2}\\
& \lesssim N_1^{-1}  (L_1 L_2 L_3 L_4)^{\frac{1}{2}} \prod_{i=1}^{4} \|u_{N_i, L_i} \|_{L^2}.
\end{align*}
Next we treat the case $\min(|\xi_1|,|\eta_1|,|\xi_2|,|\eta_2|) \ll N_1$. 
Without loss of generality, assume $|\xi_2| \ll N_1$. 
Clearly, this implies $|\xi_1| \sim N_1$ since $|\xi_3| \ll N_1$ and 
$|\xi_4| \sim N_1$. Therefore we get $|\partial_{1} \varphi(\xi_1,\eta_1) - \partial_{1} \varphi(\xi_2,\eta_2)| \gtrsim N_1^2$ which yields
\begin{equation*}
\| u_{N_1,L_1} \, u_{N_2, L_2} \|_{L_t^2 L_{xy}^2} \lesssim N_1^{-\frac{1}{2}} 
(L_1 L_2)^{\frac{1}{2}} \| u_{N_1, L_1} \|_{L_t^2 L_{xy}^2} 
\|u_{N_2, L_2}  \|_{L_t^2 L_{xy}^2}.
\end{equation*}
This and \eqref{est01-case3-1} verify the desired estimate. 

%%%%%%%%%%%%%%%%%%%%%%%%%%%%%%%%%%%%%%%%%%%%%%%%%%%%%%%%%
%%%%%%%%%%%%%%%%%%%%%%%%%%%%%%%%%%%%%%%%%%%%%%%%%%%%%%%%%
To deal with the second case, we introduce the following bilinear estimate.
\begin{prop}\label{bilinear-transversal2}
Let $A$ be dyadic such that $1 \leq A \leq N_1 L_{\max}^{-8\e}$. 
Suppose that
\begin{align*}
& \supp \widehat{u}_{N_1,L_1} \subset G_{N_1,L_1} \cap 
\{(\tau,\xi,\eta) \, | \, |\xi|\sim N_1, \ |\eta| \sim  A^{-1}N_1\},\\ 
& \supp \widehat{v}_{N_2,L_2} \subset G_{N_2,L_2} \cap 
\{(\tau,\xi,\eta) \, | \, |\xi|\sim N_1, \ |\eta|\ll  A^{-1} N_1\},
\end{align*}
Then we have
\begin{equation}
\| u_{N_1,L_1} \, v_{N_2, L_2} \|_{L_t^2 L_{xy}^2} \lesssim A^{\frac{1}{4}} N_1^{-\frac{1}{2}}
(L_1 L_2)^{\frac{1}{2}} \| u_{N_1, L_1} \|_{L_t^2 L_{xy}^2} 
\|v_{N_2, L_2}  \|_{L_t^2 L_{xy}^2}.\label{bilinear-trans01-case3-2}
\end{equation}
\end{prop}
%%%%%%%%%%%%%%%%%%%%%%%%%%%%%%%%%%%%%%%%%%%%%%%%%%%%%%%%%
%%%%%%%%%%%%%%%%%%%%%%%%%%%%%%%%%%%%%%%%%%%%%%%%%%%%%%%%%
\begin{proof}
By Plancherel's theorem, it suffices to show
\begin{equation}
\begin{split}
& \left\| \int \widehat{u}_{N_1,L_1} (\tau_1,\xi_1,\eta_1) 
\widehat{v}_{N_2, L_2}(\tau-\tau_1,\xi-\xi_1,\eta-\eta_1) d\sigma_1\right\|_{L_{\tau}^2 L_{\xi \eta}^2}\\
& \lesssim A^{\frac{1}{4}} N_1^{-\frac{1}{2}}
(L_1 L_2)^{\frac{1}{2}} \| u_{N_1, L_1} \|_{L_t^2 L_{xy}^2} 
\|v_{N_2, L_2}  \|_{L_t^2 L_{xy}^2},\label{bilinear-trans02-case3-2}
\end{split}
\end{equation}
where $\sigma_1 = d \tau_1 d \xi_1 d \eta_1$. 
The proof is divided into two cases. $|\xi_1^2 - (\xi - \xi_1)^2| \gg A^{-3/2} N_1^2$ and $|\xi_1^2 - (\xi - \xi_1)^2| \lesssim A^{-3/2} N_1^2$. Note that the latter condition means that either $|\xi_1+(\xi-\xi_1)|\lesssim A^{-3/2}N_1$ or $|\xi_1-(\xi-\xi_1)|\lesssim A^{-3/2}N_1$ holds. Thus, by the almost orthogonality, we can assume that $\xi_1$ is confined to an interval whose length is $A^{-3/2} N_1$. 
By following a standard argument, we observe that
\begin{align*}
\biggl\| \int  \ha{u}_{N_1, L_1} (\tau_1,\xi_1,\eta_1)   &
\ha{v}_{N_2, L_2}  (\tau-\tau_1,\xi-\xi_1,\eta-\eta_1) d\sigma_1 \biggr\|_{L^2}  \\
\leq & \left\|   \left(\left| 
\ha{u}_{N_1, L_1}  \right|^2 * 
\left|\ha{v}_{N_2, L_2}  \right|^2
 \right)^{1/2} |E(\tau, \xi,\eta )|^{1/2} \right\|_{L^2} \\
\leq & \sup_{\tau, \xi,\eta} |E(\tau, \xi,\eta )|^{1/2} 
 \left\| \left| 
 \ha{u}_{N_1, L_1} \right|^2 * 
\left|\ha{v}_{N_2, L_2}  \right|^2
\right\|_{L^1}^{1/2}\\
\leq & \sup_{\tau, \xi,\eta} |E(\tau, \xi,\eta )|^{1/2} 
\|\ha{u}_{N_1, L_1}\|_{L^2} 
\| \ha{v}_{N_2, L_2}  \|_{L^2},
\end{align*}
where $E(\tau, \xi,\eta) \subset \R^{3}$ is defined by
\begin{equation*}
E(\tau, \xi,\eta) := \{ (\tau_1, \xi_1,\eta_1) \in \supp \widehat{u}_{N_1,L_1}
\, | \, (\tau-\tau_1, \xi- \xi_1,\eta-\eta_1) \in \supp \widehat{v}_{N_2,L_2} \}.
\end{equation*}
Then it suffices to show $|E(\tau, \xi,\eta )| \lesssim A^{1/2}N_1^{-1} L_1 L_2$. 
The condition of the former case $|\xi_1^2 - (\xi - \xi_1)^2| \gg A^{-3/2} N_1$ implies 
$|\partial_{1} \varphi(\xi_1,\eta_1) - \partial_{1} \varphi(\xi-\xi_1, \eta-\eta_1)| \gtrsim A^{-3/2}N_1^2$. 
This and $|\eta_1| \sim A^{-1} N_1$ give $|E(\tau, \xi,\eta )| \lesssim A^{1/2}N_1^{-1} L_1 L_2$. 
Similarly, as above, the condition of the latter case allows us to assume that $\xi_1$ is confined to an interval whose length is $A^{-3/2} N_1$. 
This and $|\partial_2 \varphi(\xi_1,\eta_1) - \partial_2 \varphi(\xi-\xi_1, \eta-\eta_1)| \gtrsim A^{-2}N_1^2$ yield $|E(\tau, \xi,\eta )| \lesssim A^{1/2}N_1^{-1} L_1 L_2$.  
\end{proof}
%%%%%%%%%%%%%%%%%%%%%%%%%%%%%%%%%%%%%%%%%%%%%%%%%%%%%%%%%
%%%%%%%%%%%%%%%%%%%%%%%%%%%%%%%%%%%%%%%%%%%%%%%%%%%%%%%%%
\begin{proof}[Proof of \eqref{goal03-mth} under the conditions \textnormal{(2)}.]
For $1 \leq A \leq N_1 L_{\max}^{-8\e}$, assume
\begin{equation*}
A^{-1} N_1 \leq \max(|\eta_1|, |\eta_2|, |\eta_2|, |\eta_4|) \leq 2 A^{-1} N_1. 
\end{equation*}
Our goal is to establish
\begin{equation}
 \left| \int{ u_{N_1, L_1} u_{N_2, L_2} u_{N_3, L_3} u_{N_4, L_4}} dtdxdy \right| 
 \lesssim A^{\frac{1}{2}} N_1^{-1}  (L_1 L_2 L_3 L_4)^{\frac{1}{2}} \prod_{i=1}^{4} \|u_{N_i, L_i} \|_{L^2}.
\label{goal-case3-2}
\end{equation}
It is clear that \eqref{goal-case3-2} gives \eqref{goal03-mth} under the conditions (2). 
We divide the proof of \eqref{goal-case3-2} into the following three cases.

(2a) $|\eta_1| \sim |\eta_2| \sim |\eta_3| \sim |\eta_4|$,

(2b) $|\eta_1| \sim |\eta_2| \sim |\eta_3| \gg |\eta_4|$, 

(2c) $|\eta_1| \sim |\eta_2| \gg |\eta_3| \gtrsim |\eta_4|$.\\
The proofs are quite simple. The case (2a) can be treated by the following $L^4$ Strichartz estimate which is given by \eqref{Strichartz-1} with $p=q=4$.
\begin{equation}
\|u_{N_i,L_i} \|_{L_t^4 L_{xy}^4} \lesssim  A^{\frac{1}{8}}N_1^{-\frac{1}{4}}L_i^{\frac{1}{2}} \|u_{N_i,L_i}\|_{L^2}.
\label{est01-case3-2}
\end{equation}
The second case (2b) is handled by \eqref{est01-case3-2} and Proposition \ref{bilinear-transversal2}. 
To be precise, we use
\begin{align*}
& 
\|u_{N_1,L_1} \|_{L_t^4 L_{xy}^4} \lesssim  A^{\frac{1}{8}}N_1^{-\frac{1}{4}}L_1^{\frac{1}{2}} \|u_{N_1,L_1}\|_{L^2}, \ 
\|u_{N_2,L_2} \|_{L_t^4 L_{xy}^4} \lesssim  A^{\frac{1}{8}}N_1^{-\frac{1}{4}}L_2^{\frac{1}{2}} \|u_{N_2,L_2}\|_{L^2},\\
& \| u_{N_3,L_3} \, u_{N_4, L_4} \|_{L^2} \lesssim A^{\frac{1}{4}} N_1^{-\frac{1}{2}}
(L_3 L_4)^{\frac{1}{2}} \| u_{N_3, L_3} \|_{L^2}
\|u_{N_4, L_4} \|_{L^2}.
\end{align*}
For the last case, we employ Proposition \ref{bilinear-transversal2} which provides
\begin{align*}
& \| u_{N_1,L_1} \, u_{N_3, L_3} \|_{L^2} \lesssim A^{\frac{1}{4}} N_1^{-\frac{1}{2}}
(L_1 L_3)^{\frac{1}{2}} \| u_{N_1, L_1} \|_{L^2}
\|u_{N_3, L_3} \|_{L^2},\\
& \| u_{N_2,L_2} \, u_{N_4, L_4} \|_{L^2} \lesssim A^{\frac{1}{4}} N_1^{-\frac{1}{2}}
(L_2 L_4)^{\frac{1}{2}} \| u_{N_2, L_2} \|_{L^2}
\|u_{N_4, L_4} \|_{L^2}.
\end{align*}
These immediately establish \eqref{goal-case3-2}.
\end{proof}

%%%%%%%%%%%%%%%%%%%%%%%%%%%%%%%%%%%%%%%%%%%%%%%%%%%%%%%%%
%%%%%%%%%%%%%%%%%%%%%%%%%%%%%%%%%%%%%%%%%%%%%%%%%%%%%%%%%
We lastly consider the case (3). The following proposition plays a key role.
\begin{prop}\label{bilinear-transversal3}
Suppose that
\begin{align*}
& \supp \widehat{u}_{N_1,L_1} \subset G_{N_1,L_1} \cap 
\{(\tau,\xi,\eta) \, | \, |\xi|\sim N_1, \ |\eta| \leq L_{\max}^{8 \e}\},\\ 
& \supp \widehat{v}_{N_2,L_2} \subset G_{N_2,L_2} \cap 
\{(\tau,\xi,\eta) \, | \, |\xi|\sim N_1, \ |\eta|\leq L_{\max}^{8 \e}\},
\end{align*}
Then we have
\begin{equation}
\begin{split}
& \| u_{N_1,L_1} \, v_{N_2, L_2} \|_{L_t^2 L_{xy}^2}\\ 
& \lesssim  N_1^{-\frac{1}{4}}( L_{\max}^{-2\e}
(L_1 L_2)^{\frac{1}{2}} + L_{\max}^{10\e} \min (L_1, L_2)^{\frac{1}{2}}) \| u_{N_1, L_1} \|_{L_t^2 L_{xy}^2} 
\|v_{N_2, L_2}  \|_{L_t^2 L_{xy}^2}.\label{bilinear-trans01-case3-3}
\end{split}
\end{equation}
\end{prop}
%%%%%%%%%%%%%%%%%%%%%%%%%%%%%%%%%%%%%%%%%%%%%%%%%%%%%%%%%
%%%%%%%%%%%%%%%%%%%%%%%%%%%%%%%%%%%%%%%%%%%%%%%%%%%%%%%%%
\begin{proof}
The proof is almost the same as that for Proposition \ref{bilinear-transversal2}. We will establish
\begin{equation}
\begin{split}
& \left\| \int \widehat{u}_{N_1,L_1} (\tau_1,\xi_1,\eta_1) 
\widehat{v}_{N_2, L_2}(\tau-\tau_1,\xi-\xi_1,\eta-\eta_1) d\sigma_1\right\|_{L_{\tau}^2 L_{\xi \eta}^2}\\
& \lesssim N_1^{-\frac{1}{4}}( L_{\max}^{-2\e}
(L_1 L_2)^{\frac{1}{2}} + L_{\max}^{10\e} \min (L_1, L_2)^{\frac{1}{2}}) \| u_{N_1, L_1} \|_{L_t^2 L_{xy}^2} 
\|v_{N_2, L_2}  \|_{L_t^2 L_{xy}^2},\label{bilinear-trans02-case3-3}
\end{split}
\end{equation}
where $\sigma_1 = d \tau_1 d \xi_1 d \eta_1$. 
We consider two cases. $|\xi_1^2 - (\xi - \xi_1)^2| \gg  N_1^{1/2}L_{\max}^{12 \e}$ and $|\xi_1^2 - (\xi - \xi_1)^2| \lesssim  N_1^{1/2}L_{\max}^{12 \e}$. As we saw in the proof of Proposition \ref{bilinear-transversal2}, the latter condition means that we can assume that $\xi_1$ is confined to an interval whose length is $N_1^{-1/2}L_{\max}^{12 \e}$. 
In the same manner as in the proof of Proposition \ref{bilinear-transversal2}, we will show 
\begin{equation}
|E(\tau, \xi,\eta )| \lesssim N_1^{-\frac{1}{2}}( L_{\max}^{-4\e}
L_1 L_2 + L_{\max}^{20\e}\min (L_1, L_2))\label{est03-case3-3} 
\end{equation}
where $E(\tau, \xi,\eta) \subset \R^{3}$ is defined by
\begin{equation*}
E(\tau, \xi,\eta) := \{ (\tau_1, \xi_1,\eta_1) \in \supp \widehat{u}_{N_1,L_1}
\, | \, (\tau-\tau_1, \xi- \xi_1,\eta-\eta_1) \in \supp \widehat{v}_{N_2,L_2} \}.
\end{equation*}
For the former case, we have 
$|\partial_{1} \varphi(\xi_1,\eta_1) - \partial_{1} \varphi(\xi-\xi_1, \eta-\eta_1)| \gtrsim N_1^{1/2}L_{\max}^{12 \e}$ which, combined with $|\eta_1| \leq L_{\max}^{8 \e}$, 
gives $|E(\tau, \xi,\eta )| \lesssim N_1^{-1/2} L_{\max}^{-4\e} L_1 L_2$. 
The latter term can be handled by the inequalities $\mathcal{O}(\xi_1) \lesssim N_1^{-1/2}L_{\max}^{12 \e}$ and $|\eta_1| \leq L_{\max}^{8 \e}$, which yield 
$|E(\tau, \xi,\eta )| \lesssim N_1^{-1/2}L_{\max}^{20\e}\min (L_1, L_2)$.
\end{proof}
%%%%%%%%%%%%%%%%%%%%%%%%%%%%%%%%%%%%%%%%%%%%%%%%%%%%%%%%%
%%%%%%%%%%%%%%%%%%%%%%%%%%%%%%%%%%%%%%%%%%%%%%%%%%%%%%%%%
\begin{proof}[Proof of \eqref{goal03-mth} under the conditions \textnormal{(3)}.]
Suppose that $\e >0$ is sufficient small. By using Proposition \ref{bilinear-transversal3}, we easily obtain
\begin{align*}
\left| \int{ u_{N_1, L_1} u_{N_2, L_2} u_{N_3, L_3} u_{N_4, L_4}} dtdxdy \right| & \lesssim \|u_{N_1,L_1} u_{N_2, L_2}\|_{L_t^2 L_{xy}^2} \|u_{N_3,L_3}u_{N_4,L_4}\|_{L_t^2 L_{xy}^2}\\
& \lesssim N_1^{-\frac{1}{2}} (L_1 L_2 L_3 L_4)^{\frac{1}{2}-\e} \prod_{i=1}^{4} \|u_{N_i, L_i} \|_{L^2}.
\end{align*}
This completes the proof of \eqref{goal03-mth}.
\end{proof}
%%%%%%%%%%%%%%%%%%%%%%%%%%%%%%%%%%%%%%%%%%%%%%%%%%%%%%%%%
%%%%%%%%%%%%%%%%%%%%%%%%%%%%%%%%%%%%%%%%%%%%%%%%%%%%%%%%%
%%%%%%%%%%%%%%%%%%%%%%%%%%%%%%%%%%%%%%%%%%%%%%%%%%%%%%%%%
%%%%%%%%%%%%%%%%%%%%%%%%%%%%%%%%%%%%%%%%%%%%%%%%%%%%%%%%%
\section{Preliminaries for Theorem \ref{mth2}}
In this section, we collect notations and estimates that we utilize in the proof of the key estimate for 
Theorem \ref{mth2}. 
We begin with the definitions of $U^p$ and $V^p$ spaces which were exploited in \cite{KT05}. 
The definitions and notations of $U^p$ and $V^p$ correspond to \cite{HHK09} and \cite{HHK09-2}. 
Let $u=u(t,x)$ with $(t,x) = (t,x_1,\ldots, x_d) \in \R \times \R^d$. 
$\F_t u$ and $\F_{x} u$ denote the Fourier transform of $u$ in time and space, respectively. 
$\F_{t, x} u = \ha{u}$ denotes the Fourier transform of $u$ in space and time. 
We define frequency and modulation projections $P_N$, $Q_L$ as
\begin{align*}
(\F_{x}^{-1}P_{N}u )(\xi ):= & \psi_{N}(|\xi| )(\F_{x}{u})(\xi ),\\
\widehat{Q_{L} u}(\tau ,\xi ):= & \psi_{L}(\tau - \xi_1 |\xi|^2)\widehat{u}(\tau ,\xi ),
\end{align*}
where $(\tau,\xi)=(\tau,\xi_1,\ldots,\xi_d) \in \R \times \R^d$ are time and space frequencies. 
Let $\mathcal{Z}$ be the set of finite partitions 
$-\infty = t_0 < t_1 \cdots < t_K = \infty$ and let $\mathcal{Z}_0$ be the set of finite partitions 
$-\infty < t_0 < t_1 \cdots < t_K \leq \infty$ and let $\mathcal{Z}_0$. 
We first define $U^p$ space. 
\begin{defn}[Definition 2.1. \cite{HHK09}] 
Let $1 \leq p < \infty$. 
For $\{t_k \}_{k=0}^K \in \mathcal{Z}$ and $\{ \phi_k \}_{k=0}^{K-1} \subset L^2$ with 
$\sum_{k=0}^{K-1} \| \phi_k \|_{L^2}^p =1$ and $\phi_0 =0$ we call the function $a : \R \to L^2$ given by
\begin{equation*}
a = \sum_{k=1}^K \chi_{[t_{k-1},t_k)} \phi_{k-1}
\end{equation*}
a $U^p$-atom. 
Furthermore, we define the atomic space
\begin{equation*}
U^p := \left\{ u = \sum_{j=1}^{\infty} \lambda_j a_j \, \biggl| \, a_j : U^p \textnormal{-atom}, \ 
\lambda_j \in \C \ \textnormal{such that } \sum_{j=1}^\infty |\lambda_j| < \infty \right\} 
\end{equation*}
with norm
\begin{equation*}
\| u \|_{U^p} := \inf \left\{ \sum_{j=1}^\infty |\lambda_j| \, \biggl| \, u = \sum_{j=1}^{\infty} \lambda_j a_j , \ 
\lambda_j \in \C, \ a_j : U^p \textnormal{-atom} \right\}.
\end{equation*}
\end{defn}
%%%%%%%%%%%%%%%%%%%%%%%%%%%%%%%%%%%%%%%%%%%%%%%%%%%%%%%%%
%%%%%%%%%%%%%%%%%%%%%%%%%%%%%%%%%%%%%%%%%%%%%%%%%%%%%%%%%
%%%%%%%%%%%%%%%%%%%%%%%%%%%%%%%%%%%%%%%%%%%%%%%%%%%%%%%%%
%%%%%%%%%%%%%%%%%%%%%%%%%%%%%%%%%%%%%%%%%%%%%%%%%%%%%%%%%
%%%%%%%%%%%%%%%%%%%%%%%%%%%%%%%%%%%%%%%%%%%%%%%%%%%%%%%%%
%%%%%%%%%%%%%%%%%%%%%%%%%%%%%%%%%%%%%%%%%%%%%%%%%%%%%%%%%
Next we define $V^p$ space.
\begin{defn}[Definition 2.3. \cite{HHK09} and \textnormal{(iii)}. \cite{HHK09-2}] 
Let $1 \leq p < \infty$. 
$V^p$ space is defined as the normed space of all functions $v:\R \to L^2$ such that 
$\lim_{t \to \pm} v(t)$ exist and for which the norm
\begin{equation*}
\| v \|_{V^p} := \sup_{ \{t_k\}_{k=0}^{K} \in \mathcal{Z}} 
\left( \sum_{k=1}^K \| v(t_k) - v(t_{k-1}) \|_{L_x^2}^p \right)^{1/p}
\end{equation*}
is finite, where we use the convention that $v( - \infty)= \lim_{t \to -\infty}v(t)$ and $v (\infty) = 0$. 
Likewise, let $V^p_{-}$ denote the closed subspace of all $v \in V^p$ with $\lim_{t \to - \infty} v(t) =0$.
\end{defn}
%%%%%%%%%%%%%%%%%%%%%%%%%%%%%%%%%%%%%%%%%%%%%%%%%%%%%%%%%
%%%%%%%%%%%%%%%%%%%%%%%%%%%%%%%%%%%%%%%%%%%%%%%%%%%%%%%%%
%%%%%%%%%%%%%%%%%%%%%%%%%%%%%%%%%%%%%%%%%%%%%%%%%%%%%%%%%
For the properties of $U^p$ and $V^p$ spaces, see Propositions 2.2 and 2.4 in \cite{HHK09}, respectively. 
See also \cite{HHK09-2}. 
%%%%%%%%%%%%%%%%%%%%%%%%%%%%%%%%%%%%%%%%%%%%%%%%%%%%%%%%%
%%%%%%%%%%%%%%%%%%%%%%%%%%%%%%%%%%%%%%%%%%%%%%%%%%%%%%%%%
%%%%%%%%%%%%%%%%%%%%%%%%%%%%%%%%%%%%%%%%%%%%%%%%%%%%%%%%%

We next introduce the important connection between $U^p$ and $V^p$.
\begin{prop}[Theorem 2.8 and Proposition 2.10. \cite{HHK09}]\label{prop-duality}
Let $1<p<\infty$, $u \in V^1_{-}$ be absolutely continuous on compact intervals and $v \in V^{p'}$. 
Then,
\begin{equation*}
\| u \|_{U^p} = \sup_{\|v \|_{V^{p'}} =1} \left| \int_{-\infty}^\infty \LR{u' (t), v(t)}_{L_x^2} dt \right|.
\end{equation*}
\end{prop}
%%%%%%%%%%%%%%%%%%%%%%%%%%%%%%%%%%%%%%%%%%%%%%%%%%%%%%%%%
%%%%%%%%%%%%%%%%%%%%%%%%%%%%%%%%%%%%%%%%%%%%%%%%%%%%%%%%%
%%%%%%%%%%%%%%%%%%%%%%%%%%%%%%%%%%%%%%%%%%%%%%%%%%%%%%%%%
The following definitions correspond to Definition 2.15 in \cite{HHK09}.
\begin{defn}
Let $S = -\partial_{x_1} \Delta$. We define\\
$ \, $(i) $U_{S}^p = e^{\cdot \, S} U^p$ with norm $\|u \|_{U_S^p} = \|e^{- \, \cdot \, S} u\|_{U^p}$,\\
(ii) $V_{S}^p = e^{\cdot \, S} V^p$ with norm $\|u \|_{V_S^p} = \|e^{- \, \cdot \, S} u\|_{V^p}$.
\end{defn}
%%%%%%%%%%%%%%%%%%%%%%%%%%%%%%%%%%%%%%%%%%%%%%%%%%%%%%%%%
%%%%%%%%%%%%%%%%%%%%%%%%%%%%%%%%%%%%%%%%%%%%%%%%%%%%%%%%%
%%%%%%%%%%%%%%%%%%%%%%%%%%%%%%%%%%%%%%%%%%%%%%%%%%%%%%%%%
Now we define the solution space $Y^s$ as the closure of all 
$u \in C(\R; H^s(\R^d) \cap \LR{\nabla_x}^{-s} U^2_S$ such that
\begin{equation*}
\| u \|_{Y^s} := \left( \sum_{N} N^{2 s} \| P_N u \|_{U^2_S}^2 \right)^{1/2} < \infty
\end{equation*}
%%%%%%%%%%%%%%%%%%%%%%%%%%%%%%%%%%%%%%%%%%%%%%%%%%%%%%%%%
%%%%%%%%%%%%%%%%%%%%%%%%%%%%%%%%%%%%%%%%%%%%%%%%%%%%%%%%%
%%%%%%%%%%%%%%%%%%%%%%%%%%%%%%%%%%%%%%%%%%%%%%%%%%%%%%%%%
%%%%%%%%%%%%%%%%%%%%%%%%%%%%%%%%%%%%%%%%%%%%%%%%%%%%%%%%%
%%%%%%%%%%%%%%%%%%%%%%%%%%%%%%%%%%%%%%%%%%%%%%%%%%%%%%%%%
%%%%%%%%%%%%%%%%%%%%%%%%%%%%%%%%%%%%%%%%%%%%%%%%%%%%%%%%%
We collect the fundamental estimates of the Zakharov-Kuznetsov equation. 
The first estimate is included in Proposition 1.3 in \cite{SchippaBOZK} which was obtained in the same way as for the Strichartz estimate of a higher dimensional version of the Benjamin-Ono equation. 
See the proof of Theorem 1.1 in \cite{HLRRW}. 
It should be noted that we can get the $L^4$ Strichartz estimate below by following the proof of Theorem 2 in \cite{GPS09}. See also \cite{Gru14}.
\begin{prop}\label{l4-strichartz-3d}
Let $d\geq 3$. Then we have
\begin{equation*}
\| e^{tS} \varphi \|_{L_t^4 L_x^4} \lesssim  \||\nabla|^{\frac{d-3}{4}} \varphi \|_{L_x^2}.
\end{equation*}
\end{prop}
\begin{rem}
By using Proposition 2.19 in \cite{HHK09}, the above $L^4$ Strichartz estimate yields
\begin{equation}
\| u \|_{L_t^4 L_x^4} \lesssim  \||\nabla|^{\frac{d-3}{4}} u \|_{U^4_S}.\label{l4-strichartz-u4}
\end{equation}
\end{rem}
%%%%%%%%%%%%%%%%%%%%%%%%%%%%%%%%%%%%%%%%%%%%%%%%%%%%%%%%%
%%%%%%%%%%%%%%%%%%%%%%%%%%%%%%%%%%%%%%%%%%%%%%%%%%%%%%%%%
The following bilinear transversal estimate can be found in \cite{SchippaBOZK}. 
\begin{prop}[Proposition 1.2 in \cite{SchippaBOZK} with $a=2$]\label{bilinear-transversal-d3}
Let $d\geq 2$, $N_2 \ll N_1$ and define $u_{N_1}=P_{N_1} u$, $v_{N_2} = P_{N_2} v$. 
Then we have
\begin{equation}
\| e^{t S} u_{N_1} \, e^{t S}v_{N_2} \|_{L_t^2 L_x^2} \lesssim N_1^{-1} N_2^{\frac{d-1}{2}} 
 \| u_{N_1} \|_{L_t^2 L_x^2} \|v_{N_2}  \|_{L_t^2 L_x^2}.\label{bi-trans-l2}
\end{equation}
\end{prop}
\begin{rem}
By the same argument as of the proof of Corollary 2.21 in \cite{HHK09}, we can see that \eqref{bi-trans-l2} implies
\begin{equation}
\| u_{N_1} \, v_{N_2} \|_{L_t^2 L_x^2} \lesssim N_1^{-1} N_2^{\frac{d-1}{2}} 
 \| u_{N_1} \|_{U^2_S} \|v_{N_2}  \|_{U^2_S}.\label{bi-trans-u2}
\end{equation}
If $d \geq 3$, by interpolating the above bilinear estimate and 
\begin{equation*}
\| u_{N_1} \, v_{N_2} \|_{L_t^2 L_x^2} \lesssim (N_1 N_2)^{\frac{d-3}{2}} 
 \| u_{N_1} \|_{U^4_S} \|v_{N_2}  \|_{U^4_S},
\end{equation*}
which follows from \eqref{l4-strichartz-u4}, for any $\e >0$, we get 
\begin{equation}
\| u_{N_1} \, v_{N_2} \|_{L_t^2 L_x^2} \lesssim N_1^{-1+ \e} N_2^{\frac{d-1}{2}-\e} 
 \| u_{N_1} \|_{V^2_S} \|v_{N_2}  \|_{V^2_S}.\label{bi-trans-v2}
\end{equation}
\end{rem}
%%%%%%%%%%%%%%%%%%%%%%%%%%%%%%%%%%%%%%%%%%%%%%%%%%%%%%%%%
%%%%%%%%%%%%%%%%%%%%%%%%%%%%%%%%%%%%%%%%%%%%%%%%%%%%%%%%%
\section{Proof of key estimate for Theorem \ref{mth2}}
We show the following key estimate which immediately yields Theorem \ref{mth2}. 
We omit the proof of Theorem \ref{mth2} here and focus on the key estimate. 
To complete the proof of Theorem \ref{mth2}, see the proof of Theorem 1.1 in \cite{HHK09}.
\begin{thm}\label{goal-3d}
Let $T \in (0, \infty]$. We define the Duhamel term as
\begin{equation*}
I_T (u_1, u_2, u_3)(t) := \int_0^t \chi_{[0, T)}e^{(t-t')S} \partial_{x_1} (u_1 \, u_2 \, u_3)(t') d t'.
\end{equation*}
Then there exists $C>0$ such that
\begin{equation}
\| I_T (u_1, u_2, u_3)\|_{Y^{s_c}} \leq C \|u_1 \|_{Y^{s_c}} \| u_2 \|_{Y^{s_c}} \|u_3 \|_{Y^{s_c}}.
\end{equation}
\end{thm}
By using Proposition \ref{prop-duality} above and Proposition 2.4 in \cite{HHK09}, it suffices to show the following proposition.
\begin{prop}\label{goal2-3d}
Let $N_1 \geq N_2 \geq N_3 \geq N_4$ and $u_{N_i}=P_{N_i} u_i$ where $i=1,2,3,4$. 
Then we have
\begin{equation*}
\left| \int{ u_{N_1} u_{N_2} u_{N_3} u_{N_4}} dtdx \right| 
\lesssim N_1^{-\frac{3}{2}+\e} N_3^{\frac{d-2}{2}} N_4^{\frac{d-1}{2}-\e} \prod_{i=1}^4 \| u_{N_i} \|_{V^2_S}.
\end{equation*}
\end{prop}
%%%%%%%%%%%%%%%%%%%%%%%%%%%%%%%%%%%%%%%%%%%%%%%%%%%%%%%%%
%%%%%%%%%%%%%%%%%%%%%%%%%%%%%%%%%%%%%%%%%%%%%%%%%%%%%%%%%
\begin{proof}
We divide the proof into three cases.

Case 1: $N_1 \sim N_2 \gg N_3 \gtrsim N_4$,

Case 2: $N_1 \sim N_2 \sim N_3 \gg N_4$,

Case 3: $N_1 \sim N_2 \sim N_3 \sim N_4$.\\
\underline{Case 1} The first case can be handled by \eqref{bi-trans-v2} as follows.
\begin{align*}
\left| \int{ u_{N_1} u_{N_2} u_{N_3} u_{N_4}} dtdx \right|  & \lesssim 
\|u_{N_1} u_{N_3}\|_{L_t^2 L_{x}^2} \|u_{N_2}u_{N_4}\|_{L_t^2 L_{x}^2}\\
& \lesssim N_1^{-2 + 2\e} N_3^{\frac{d-1}{2}-\e} N_4^{\frac{d-1}{2}-\e} \prod_{i=1}^{4} \|u_{N_i} \|_{V^2_S}.
\end{align*}
\underline{Case 2} It follows from \eqref{l4-strichartz-u4} and \eqref{bi-trans-v2} that
\begin{align*}
\left| \int{ u_{N_1} u_{N_2} u_{N_3} u_{N_4}} dtdx \right|  & \lesssim 
\|u_{N_1} u_{N_4}\|_{L_t^2 L_{x}^2} \|u_{N_2} \|_{L_t^4 L_x^4} \|u_{N_3}\|_{L_t^4 L_x^4}\\
& \lesssim N_1^{\frac{d-5}{2} + \e} N_4^{\frac{d-1}{2}-\e} \prod_{i=1}^{4} \|u_{N_i} \|_{V^2_S}.
\end{align*}
\underline{Case 3} Lastly, \eqref{l4-strichartz-u4} gives
\begin{align*}
\left| \int{ u_{N_1} u_{N_2} u_{N_3} u_{N_4}} dtdx \right|  & \lesssim 
\|u_{N_1} \|_{L_t^4 L_x^4} \|u_{N_2}\|_{L_t^4 L_x^4} \|u_{N_3} \|_{L_t^4 L_x^4} \|u_{N_4}\|_{L_t^4 L_x^4}\\
& \lesssim N_1^{d-3} \prod_{i=1}^{4} \|u_{N_i} \|_{V^2_S},
\end{align*}
which completes the proof of Theorem \ref{goal2-3d}.
\end{proof}
%%%%%%%%%%%%%%%%%%%%%%%%%%%%%%%%%%%%%%%%%%%%%%%%%%%%%%%%%
%%%%%%%%%%%%%%%%%%%%%%%%%%%%%%%%%%%%%%%%%%%%%%%%%%%%%%%%%

\end{document}